\documentclass{amsart}
\usepackage{xypic}

\xyoption{all}

\newif\iffurther
\furthertrue

\newtheorem{thm}{Theorem}[section] 

\newtheorem{claim}[thm]{Claim}

\newtheorem{cor}[thm]{Corollary}

\newtheorem{exmpl}[thm]{Example}
\newtheorem{lem}[thm]{Lemma}
\newtheorem{prop}[thm]{Proposition}

\newtheorem{proper}[thm]{Property}

\newtheorem{ques}[thm]{Question}

\def\[{\left[}
\def\]{\right]}

\def\GK{{\operatorname{GKdim}}}

\def\GKdim{{Gel'fand-Kirillov dimension}}

\long\def\forget#1\forgotten{{}}

\begin{document}

\title{Prime and Primitive Alegbras with Prescribed Growth Types}

\author{Be'eri Greenfeld}
\address{Department of Mathematics, Bar Ilan University, Ramat Gan 5290002, Israel}
\email{beeri.greenfeld@gmail.com}

\thanks{The writer wishes to thank Prof.~Agata Smoktunowicz and Prof.~Rostislav Grigorchuk for interesting related correspondence, and Prof.~Uzi Vishne and Prof.~Louis Rowen for their remarks about the paper. The writer is gratefully thankful to the referee.}

\begin{abstract}
Bartholdi and Smoktunowicz constructed in \cite{SmoktunowiczBartholdi} finitely generated monomial algebras with prescribed sufficiently fast growth types. We show that their construction need not result in a prime algebra, but it can be modified to provide prime algebras without further limitations on the growth type. 

Moreover, using a construction of an inverse system of monomial ideals which arise from this construction we are able to further construct finitely generated primitive algebras without further limitations on the growth type.

Then, inspired by Zelmanov's example \cite{Zelmanov}, we show how our prime algebras can be constructed such that they contain non-zero locally nilpotent ideals; this is the very opposite of the primitive constructions.

\end{abstract}

\maketitle

\section{Introduction} \label{section1}

\subsection{Background} A major direction in the investigation of growth of algebras is the problem of realizing functions as the growth types of finitely generated algebras. Clearly, every such growth function is increasing and submultiplicative, namely $f(n)<f(n+1)$ and $f(m+n)\leq f(m)f(n)$.

Since the precise growth function is determined by the generating subspace one chooses, it is defined up to the equivalence $f\sim g$ if and only if $f(n)\leq g(Cn)\leq f(Dn)$ for some $C,D>0$ and for all $n\in \mathbb{N}$. For further information about growth functions we refer the reader to \cite{delaHarpe}.

Fix an arbitrary base field $F$. 
In \cite{SmoktunowiczBartholdi}, the authors prove the following:

\begin{thm}[{\cite[Theorem~C]{SmoktunowiczBartholdi}}] \label{Corollary}
Let $f:\mathbb{N}\rightarrow \mathbb{N}$ be submultiplicative and increasing. Then there exists a finitely generated monomial algebra $B$ whose growth function satisfies $$f(2^n)\leq \dim_F B(2^n)\leq 2^{2n+3}f(2^{n+1}).$$
\end{thm}

They deduce the following corollary:

\begin{cor}[{\cite[Corollary~D]{SmoktunowiczBartholdi}}]
Let $f:\mathbb{N}\rightarrow \mathbb{N}$ be a submultiplicative, increasing, and such that $f(Cn)\geq nf(n)$ for some $C>0$ and all $n\in \mathbb{N}$. Then there exists an associative algebra of growth $\sim f$.
\end{cor}

The authors mention that they do not know whether their construction results in prime algebras. This paper in particular answers this, by exhibiting a similar procedure which allows one to construct primitive algebras with prescribed growth.

\begin{thm} \label{MainResult1}
Let $f:\mathbb{N}\rightarrow \mathbb{N}$ be a monotone, submultiplicative function such that there exists some $C>0$ for which $f(Cn)\geq nf(n)$.

Then there exists a finitely generated, primitive monomial algebra with growth type equivalent to $f$.
\end{thm}

Moreover, we show that it is possible to construct prime algebras with non-zero locally nilpotent ideals having prescribed growth.
This is inspired from Zelmanov's example \cite{Zelmanov}.

\begin{thm} \label{MainResult2}
Let $f:\mathbb{N}\rightarrow \mathbb{N}$ be a monotone, submultiplicative function such that there exists some $C>0$ for which $f(Cn)\geq nf(n)$.

Then there exists a finitely generated, prime monomial algebra which contains a non-zero locally nilpotent ideal and has growth type equivalent to $f$.
\end{thm}

Note that an earlier result of Trofimov \cite{Semigroup} shows that for every $f_{-}(n)\succ n^2$ and $f_{+}\prec \exp(n)$ there exists a $2$-generated semigroup with growth function infinitely often smaller than $f_{-}$ and infinitely often larger than $f_{+}$.

Before proceeding to recall the original construction, we mention that the situation for growth of groups is very different. Grigorchuk \cite{gap1,gap2} conjectured that there do not exist group growth types which are super-polynomial but strictly asymptotically lower than $e^{n^\beta}$ (for some $0<\beta<1$).

\subsection{Aim and structure of the paper}
We show in Section \ref{section2} that the construction from \cite{SmoktunowiczBartholdi} can be modified such that the resulting algebras are prime (though in general the resulting algebras need not be prime).

In Section \ref{sectionprimitive} we construct primitive algebras with prescribed growth types. Namely, we prove Theorem \ref{MainResult1}.
This is done through a construction of inductive systems of monomial ideals, which enables us to achieve sufficient control both on the growth of the constructed algebras and on their Jacobson radical, from which we deduce primitivity.

In Section \ref{sectionnilpotent} we construct prime algebras with non-zero locally nilpotent ideals and prescribed growth, namely, we prove Theorem \ref{MainResult2}.

\section{Prime modification} \label{section2}

We start by introducing the original construction of Bartholdi and Smoktunowicz.

\subsection{The construction from \cite{SmoktunowiczBartholdi}}
Assume $f(1)=d$. Let $A=F\left<x_1,\dots,x_d\right>$ be the $d$-generated free $F$-algebra, and recall that it is naturally graded via $\deg(x_i)=1$.

Let $M(n)$ be the set of all monomials of degree $n$ and $M=\bigcup_{n\in \mathbb{N}} M(n)$ the set of all monomials.

We inductively construct sets $W(2^n)\subseteq M(2^n)$ as follows: set $W(1)=M(1)=\{x_1,\dots,x_d\}$. To construct $W(2^{n+1})$, assume $W(2^n)$ has been constructed and pick an arbitrary subset $C(2^n)\subseteq W(2^n)$ such that $|C(2^n)|=\lceil \frac{f(2^{n+1})}{f(2^n)} \rceil$, and set $W(2^{n+1})=C(2^n)W(2^n)$.

It is indeed possible to pick such sets:
\begin{lem}
For all $n\in \mathbb{N}$, we have that $|W(2^n)|\geq \frac{f(2^{n+1})}{f(2^n)}$.
\end{lem}
\begin{proof}
For all $n\in \mathbb{N}$, we have $|W(2^n)|\geq f(2^n)$ as can be shown by induction. Now by submultiplicativity $f(2^{n+1})\leq f(2^n)^2$.
\end{proof}

Let $I\triangleleft A$ be the monomial ideal generated by all monomials which do not appear as subwords of any monomial from $W=\bigcup_{n\in \mathbb{N}} W(n)$. Namely, $I$ is generated by the following set of monomials: $\{w\in M|AwA\cap W=\emptyset\}$.

By Lemmata 6.4,6.5 from \cite{SmoktunowiczBartholdi}, the algebra $B=A/I$ has the desired growth rate.

\subsection{Prime Algebras}

We first introduce an example which shows that the original construction from \cite{SmoktunowiczBartholdi} might result in non-prime algebras.

\begin{exmpl} (A non-semiprime example.) \label{example}
Suppose $f(n)\leq (d-1)^{n}$ for all $n\in \mathbb{N}$. Then one can choose $W(1)=\{x_1,\dots,x_d\}$ and let $W(1)^*=W(1)\setminus \{x_1\}$. Now take $C(1)\subseteq W(1)^*, C(2)\subseteq C(1)\cdot W(1)^*, C(4)\subseteq C(2)\cdot C(1)\cdot W(1)^*$ and so on. This is indeed possible since we assume $f(n)\leq (d-1)^n$.

In the algebra $B$, the image of the element $x_1$ is non-zero, yet generates a nilpotent ideal, as none of the sets $C(2^n)$ contains a monomial with $x_1$ as a subword.

\end{exmpl}

We show how one can choose the sets $C(2^n)$ such that $B$ is prime.

For any $m'\geq m$, define a function $\Pi^{m'}_{m}:M(2^{m'})\rightarrow M(2^m)$ by $\Pi^{m'}_{m}(uw)=w$ where $|u|=2^{m'}-2^m,|w|=2^m$. For a subset $S\subseteq W(2^{m'})$ define $\Pi^{m'}_{m}(S)=\{\Pi^{m'}_{m}(s):s\in S\}$.

\begin{lem} \label{Pi_onto}
For every $m'>m$, the function $\Pi^{m'}_{m}$ maps $W(2^{m'})$ onto $W(2^m)$.
\end{lem}
\begin{proof}
This follows because $W(2^{m'})=C(2^{m'-1})\cdots C(2^{m})\cdot W(2^m)$.
\end{proof}

Let $\mu:\mathbb{N}\cup \{0\}\rightarrow \mathbb{N}$ be an increasing monotone function satisfying that for all $n\in \mathbb{N}\cup \{0\}$, we have $f(2^{\mu(n)+1})\geq |W(2^n)|\cdot f(2^{\mu(n)})$.

This is indeed possible, since otherwise $f(2^{d+1})\leq \alpha f(2^d)$ for some $\alpha>0$ and for all $d\in \mathbb{N}$, and therefore $f(2^m\cdot 2^d)\leq \alpha^m f(2^d)$. Taking $m$ such that $2^m>C$ (for $C$ from \ref{Corollary}), one takes $n=2^d>\alpha^m$ and obtains by the assumption $f(Cn)\geq nf(n)$:
$$ 2^d f(2^d)\leq f(C\cdot 2^d)\leq f(2^m\cdot 2^d)\leq \alpha^mf(2^d)<2^d f(2^d),$$ 
a contradiction.

For all $n\in \mathbb{N}$, choose $C(2^{\mu(n)})=S$, where $S$ is an arbitrary subset of $W(2^{\mu(n)})$ such that $\Pi^{\mu(n)}_{n}(S)=W(2^n)$. This is indeed possible by Lemma \ref{Pi_onto} and since we have, by construction of $\mu$:
$$ |C(2^{\mu(n)})|\geq \frac{f(2^{\mu(n)+1})}{f(2^{\mu(n)})}\geq |W(2^n)|.$$

\begin{prop} \label{prime}
The algebra $B$ constructed above is prime.
\end{prop}
\begin{proof}
Since $B$ is a monomial algebra we need to show that for all non-zero $0\neq u,u'\in B$ monomials, we have $uBu'\neq 0$.
Let $u$ be a subword of some monomial $w\in W(2^m)$ and $u'$ a subword of some monomial $w'\in W(2^{m'})$. Let $n\geq m,m'$, then $W(2^{\mu(n)+1})=C(2^{\mu(n)})\cdot W(2^{\mu(n)})$; since $\mu(n)>n\geq m'$, we have that there exists $v'\in W(2^{\mu(n)})$ with $w'$ a subword of $v'$ (by Lemma \ref{Pi_onto}), and by construction there exists $v\in C(2^{\mu(n)})$ with $w$ a subword of $v$. Therefore, $$vv'\in W(2^{\mu(n)+1})\cap BwBw'B\subseteq W(2^{\mu(n)+1})\cap BuBu'B$$ and hence $uBu'\neq 0$.
\end{proof}

Recall that the \textit{entropy} of a finitely graded algebra $R=\bigoplus_{n\in \mathbb{N}} R_n$ (namely, a graded algebra with all homogeneous components finite dimensional) is defined by: $$H(R)=\limsup_{n\rightarrow \infty} (\dim_F R_n)^{\frac{1}{n}}.$$

Note that the entropy of a growth function is not invariant under the equivalence relation $f\sim g\Leftrightarrow \exists\ C,D>0:\ f(n)\leq g(Cn)\leq f(Dn)$.

In \cite{entropy}, it is mentioned that by Theorem 1.1 (from \cite{entropy}) there exist graded algebras with arbitrarily small entropy.

As a byproduct of the above construction, we can improve this result for prime algebras:

\begin{cor}
For every $\varepsilon>0$ there exists a finitely generated, prime monomial algebra $R$ with $1<H(R)<1+\varepsilon$.
\end{cor}

\begin{proof}
Take $f(n)=\lceil (1+\varepsilon)^n\rceil$ in the above construction. Note that by the above construction and Theorem \ref{Corollary}, we have a finitely generated prime monomial algebra $B$ with growth function $g(n)=\dim_F B(n)$ satisfying: $$f(2^n)\leq g(2^n)\leq 2^{2n+3}f(2^n).$$

By considering the largest power of two which does not exceed $n$, it follows that for all $n\in \mathbb{N}$ we have $f(n)\leq g(2n)\leq Cn^2f(4n)$ and therefore $$\sqrt{1+\varepsilon}\leq \limsup_{n\rightarrow \infty} f(n)^{\frac{1}{2n}} \leq H(B) \leq \limsup_{n\rightarrow \infty} f(n)^{\frac{2}{n}}\leq (1+\varepsilon)^2,$$ and the result follows.

\end{proof}

\section{Primitive Algebras} \label{sectionprimitive}

\subsection{Goals and strategy} Recall that an algebra is said to be (left) \textit{primitive} if it admits a faithful simple (left) module; every simple algebra is primitive, and every primitive algebra is prime. In \cite{arbitrary}, Vishne constructed for every real number $\beta\geq 2$ a finitely generated primitive algebra with \GKdim\ $\beta$.

Our aim in this section is to show how building on the modification from Section \ref{section2}, one can construct primitive algebras over arbitrary base fields with many growth types. In particular, we obtain examples of primitive algebras of intermediate growth.

The construction goes as follows. We use the construction from Section \ref{section2} to construct an inductive system of monomial ideals in the free algebra, whose intersection is prime. Our goals are to achieve sufficient control both on the growth of the intersection and on its Jacobson radical; the latter enables us to show that the intersection is in fact a primitive ideal.

The resulting algebra is finally the quotient of the free algebra by the mentioned above intersection, and therefore a dense subalgebra of the inverse limit of the corresponding monomial algebras.

Let $f$ be a monotone increasing, submultiplicative function which satisfies $f(Cn)>nf(n)$ for some $C>0$. Without loss of generality, we may assume further the following regularity condition: $$\liminf_{n\rightarrow \infty} \frac{f(2^{n+1})}{f(2^n)}=\infty.$$

The reason that it is indeed possible to assume this condition is the following. Consider a function $f$ such that $f(Cn)\geq nf(n)$. Pick $t\in \mathbb{N}$ for which $2^t\geq C$; then $f(2^tn)\geq nf(n)$.

Now set $f'(n)=\sum_{i=0}^{t-1} n^{-\frac{i}{t}}f(2^i n)$. Observe that $f'\sim f$.
On the other hand $f'(2n)=\sum_{i=0}^{t-1} (2n)^{-\frac{i}{t}}f(2^{i+1}n)\geq \frac{1}{2} \sum_{i=1}^{t} n^{-\frac{i-1}{t}}f(2^i n)=\frac{1}{2} n^{1/t}\sum_{i=1}^{t} n^{-\frac{i}{t}}f(2^i n) \geq \frac{1}{2} n^{1/t}\sum_{i=0}^{t-1} n^{-\frac{i}{t}}f(2^i n)+(n^{-1}f(2^t n)-f(n))$, but $n^{-1}f(2^t n)\geq n^{-1}f(Cn)\geq f(n)$ so $\liminf_{n\rightarrow \infty} \frac{f'(2n)}{f'(n)}=\infty$.

Construct sets $C(2^n)$ and $W(2^n)$ with respect to $f$ as in Section \ref{section1}, namely, the original construction from \cite{SmoktunowiczBartholdi} with resulting algebra $B=A/\{w:AwA\cap W=\emptyset\}$ (where $W=\bigcup_{n\geq 0} W(2^n)$).

\subsection{Monomial ideals} We start by constructing the system of ideals, and this is done through an operator on monomial ideals.

Given a system of sets $C(2^n),W(2^n)$ such that $C(2^n)\subseteq W(2^n)$ and $W(2^{n+1})=C(2^n)W(2^n)$ such that $\liminf_{n\rightarrow\infty} |C(2^n)|=\infty$ we show how one can form a new system $C'(2^n),W'(2^n)$ extending the original system (namely: $W(2^n)\subseteq W'(2^n)$) with the same properties. Moreover, the quotient $B'=A/\{w:AwA\cap W'=\emptyset\}$ (where $W'=\bigcup_{n\geq 0} W'(2^n)$) has the following property.

\begin{proper} \label{property}
For every monomial $w\in W(2^n)$ there exists a monomial $v$ such that $w\preceq v$ and $v^t\in \bigcup_{i\geq 0} C'(2^i)$ for infinitely many $t$.
\end{proper}

The construction depends on a parameter $\varepsilon>0$, which we fix for the rest of the subsection.

For $i\geq 0$ define $\theta(i)$ such that for all $m\geq \theta(i)$, we have that $$\sum_{j\leq 2^{i}} d^j \leq \varepsilon\frac{f(2^{m+1})}{f(2^m)}.$$ This is indeed possible by the regularity assumption on $f$.

For every $w\in W(2^i)$ and $i\in \mathbb{N}\cup \{0\}$ fix some monomial $w\preceq v_w\in W(2^{\theta(i)})$.

Define sets as follows: $$X_n=\{v_w^{2^t}:w\in W(2^i)\ with\ \theta(i)\leq n; t=n-\theta(i)\}.$$

\begin{lem} \label{sizeX}
For all $n$ we have $$|X_n|\leq \sum_{j\leq 2^{i}} d^j.$$
\end{lem}
\begin{proof}
There exist only $d^j$ monomials of length $j$, and $W(2^i)$ contains monomials of length $2^i$, so their subwords can have lengths only $j\leq 2^i$.
\end{proof}

Note that $|v_w^{2^t}|=2^t|v_w|=2^{t+\theta(i)}=2^n$ and  therefore $X_n\subseteq M(2^n)$.
Inductively construct sets $C'(2^n)=C(2^n)\cup X_n$ and $W'(2^{n+1})=C'(2^n)W'(2^n)$.

Fix $n$ and pick some $v_w^{2^t}$. Then $w\in W(2^i)$ for some $i$, and $v_w\in W(2^{\theta(i)})$ and also $v_w\in X_{\theta(i)}\subseteq C'(2^{\theta(i)})$ so $v_wv_w\in W'(2^{\theta(i)}+1)$ and inductively $v_w^{2^t}\in W'(2^n)$.

It follows that $X_{n}\subseteq W'(2^{n})$, so the construction is indeed a specific case of the original construction from \cite{SmoktunowiczBartholdi}, exposed in Section \ref{section1}.

Notice that by Lemma \ref{sizeX}, $|C'(2^n)|\leq |X_n|+|C(2^n)|\leq (1+\varepsilon)|C(2^n)|$ and therefore $|W'(2^n)|\leq (1+\varepsilon)^n|W(2^n)|$.

Observe that $B'$ indeed satisfies Property \ref{property}. Take $w$ a subword of some monomial from $W(2^n)$. Then $w\preceq v_w$ and $v_w^{2^t}\in \bigcup_{i\geq 0} C'(2^i)$.

\subsection{Our construction}
We now construct the sequence of monomial ideals. Repeat the construction $$W\rightsquigarrow W'\rightsquigarrow W''\rightsquigarrow \cdots$$ with a sequence of parameters $\varepsilon_i$ such that $\prod_{i=1}^{\infty} (1+\varepsilon_i)\leq 2$. Denote $W^{(0)}=W$ and $W^{(\alpha+1)}=(W^{(\alpha)})'$. Since each $W^{(\alpha)}$ is contained in $W^{(\alpha+1)}$ we can form the union $\hat{W}$ with quotient $\hat{B}=A/\{w:AwA\cap \hat{W}=\emptyset\}$. Let $\hat{W}(2^n)$ be the set of all length $2^n$ monomials of $\hat{W}$. Similarly define $\hat{C}(2^n)$.

In the language of the corresponding monomial algebras, we have an inverse system: $$\cdots \rightarrow B''\rightarrow B'\rightarrow B.$$

We can calculate the growth of the resulting algebra $\hat{B}$, corresponding to the ideal $\hat{W}$. We have that $|\hat{W}(2^n)|\leq \left(\prod_{i=1}^{\infty} (1+\varepsilon_i)\right)^n |W(2^n)|\leq 2^n|W(2^n)|$ and $|\hat{C}(2^n)|\leq 2|C(2^n)|$. By the argumentation of Subsection \ref{growth} we have that $$\dim_F \hat{B}(2^n)\leq 2^{3n+4}f(2^{n+1})\sim f(2^n)$$ and it follows that the growth of $\hat{B}$ is asymptotically $f$.

\subsection{Primitivity of $\hat{B}$} To show that the algebras constructed above are prime, we need to show by the same argument as of Proposition \ref{prime} that for every monomial $w\preceq \hat{W}(2^n)$ there exists some $n'>n$ with $w\preceq \hat{C}(2^{n'})$.

This follows since for every $w\preceq W^{(\alpha)}(2^n)$ we have that $w\preceq C^{(\alpha+1)}(2^{n'})\subseteq \hat{C}(2^{n'})$ for $n'$ arbitrarily large, by Property \ref{property}.

To compute the Jacobson radical of $\hat{B}$ we first observe that for every non-zero monomial $w\in \hat{B}$, there exists some non-nilpotent monomial $v\in \hat{B}w\hat{B}$. This follows by considering $v=v_w$ and from Property \ref{property}, as $v_w^{2^t}\neq 0$ for arbitrarily large $t$.

By \cite{locNilp}, the Jacobsn radical of a finitely generated monomial algebra is always locally nilpotent. Assume $J\neq 0$ is the Jacobson radical of $\hat{B}$, and pick some $0\neq w_1+\cdots+w_d\in J$ with $w_1,\dots,w_d$ non-zero monomials and $d\geq 1$. We may assume $|w_1|=\cdots=|w_d|$ since $J$ is homogeneous by a classical result of Bergman (see \cite{Bergman}).

By the following claim we get a contradiction to the assumption, from which we deduce that $J=0$.

\begin{claim}
The element $uw_1u'$ is nilpotent for all $u,u'$ monomials.
\end{claim}

\begin{proof}
Since $J$ is locally nilpotent, it follows that for every monomials $u,u'\in \hat{B}$ it is the case that $uw_1u'+\cdots uw_du'$ is nilpotent of some degree $t$. Compute:
$$0=(uw_1u'+\cdots uw_du')^t=(uw_1u')^t+\sum_{i_1,\dots,i_t:\exists j: i_j\neq 1} uw_{i_1}u'uw_{i_2}u'\cdots uw_{i_t}u'.$$
Since each of the monomials in the sum on the right has some $w_{i_j}\neq w_1$ and all monomials $w_i$ have the same length it follows that the monomial $(uw_1u')^t$ is not cancelled by any monomial from the sum and since $\hat{B}$ is a monomial algebra, it follows that $(uw_1u')^t=0$.
\end{proof}

It follows that $\hat{B}$ is semiprimitive. By \cite{Okninski}, a finitely generated prime semiprimitive monomial algebra is either PI or primitive. Since $\GK(\hat{B})=\infty$ (as $f$ is super-polynomial), it is impossible that $\hat{B}$ is PI, and therefore. It follows that $\hat{B}$ is primitive as required.

\section{Locally Nilpotent Ideals} \label{sectionnilpotent}

In \cite{Zelmanov}, Zelmanov constructed a finitely generated prime monomial algebra with a non-zero locally nilpotent ideal. In \cite{SmoktunowiczVishne}, Smoktunowicz and Vishne adjusted Zelmanov's example to have quadratic growth.

The aim of this section to modify the construction from Section \ref{section2} to obtain prime algebras with non-zero locally nilpotent ideals (with growth functions arbitrary as above). Let $\left<x_1,\dots,x_m\right>$ denote the free semigroup generated by $x_1,\dots,x_m$. Let $A=F\left<x_1,\dots,x_{d+1}\right>$.
Fix $\mu:\mathbb{N}\cup\{0\}\rightarrow \mathbb{N}$ as above, and furthermore assume $\mu(n)=2^{\nu(n)}$ for every $n\in \mathbb{N}\cup\{0\}$, and $\nu$ some increasing monotone function that will be determined in the sequel.

\subsection{Construction} \label{construction}
First take $\tilde{f}(1)=f(1)+1=d+1$ and $\tilde{f}(n)=f(n)$ for all $n>1$. Construct sets $C(2^n)$ and $W(2^n)$ as in Example \ref{example} from Section \ref{section2} with respect to $\tilde{f}$, and such that $C(2^n)\subseteq \left<x_1,\dots,x_d\right>$ (but $W(1)=\{x_1,\dots,x_{d+1}\}$).

Now inductively construct sets $\tilde{C}(2^n)$ and $\tilde{W}(2^n)$. For $n\neq 2^{\nu(i)}$, set $\tilde{C}(2^n)=C(2^n)$ and $\tilde{W}(2^{n+1})=\tilde{C}(2^n)\cdot \tilde{W}(2^n)$. For $n=2^{\nu(i)}$ set $\tilde{C}(2^{2^{\nu(i)}})=C(2^{2^{\nu(i)}})\cup X_{i}$ where $X_i$ is an arbitrary subset of $\tilde{W}(2^{2^{\nu(i)}})$ of cardinality $|\tilde{W}(2^i)|$ such that $\Pi^{2^{\nu(i)}}_{i}(X_i)=\tilde{W}(2^i)$ and $\tilde{W}(2^{n+1})=\tilde{C}(2^n)\cdot \tilde{W}(2^n)$. In each step, take $\nu(i)$ arbitrary large enough such that $|\tilde{C}(2^{2^{\nu(i)}})|\leq (1+\varepsilon_i)|C(2^{2^{\nu(i)}})|$, with $\varepsilon_i$ taken such that $\prod_{i=1}^{\infty} (1+\varepsilon_i)=l<\infty$. This is indeed possible since as we mention in Section \ref{section2}, the numbers $\frac{f(2^{m+1})}{f(2^m)}$ are unbounded. It follows that $|\tilde{C}(2^n)|\leq l|C(2^n)|,|\tilde{W}(2^n)|\leq l|W(2^n)|$. 

Let $\tilde{W}=\bigcup_{n\in \mathbb{N}} \tilde{W}(2^n)$.

It follows that the algebra $B$ constructed by $B=A/\{w:AwA\cap \tilde{W}=\emptyset\}$ is prime, by the same argument of Proposition \ref{prime}: Indeed, $\Pi^{\mu(n)}_{n}(\tilde{W}(2^{\mu(n)}))=\tilde{W}(2^n)$. It naturally homomorphically maps onto the algebra $B_0=A/\{w:AwA\cap W=\emptyset\}$ (in which, as in Example \ref{example}, the ideal generated by $x_{d+1}$ is nilpotent), as $W\subseteq \tilde{W}$.

\subsection{Our construction has growth $\sim f$} \label{growth}
By the last sentence in the proof of Theorem C from \cite{SmoktunowiczBartholdi}, we have that: $$\dim_F B(2^n)\leq 2(2^n + 1)|\tilde{W}(2^n)|\cdot|\tilde{C}(2^n)|,$$
and
$$|W(2^n)|\cdot|C(2^n)|<2^{n+1}f(2^{n+1}).$$
On the other hand, by the description of the construction in Subsection \ref{construction}, we have that $|\tilde{W}(2^n)|\cdot|\tilde{C}(2^n)|\leq l^2|W(2^n)|\cdot|C(2^n)|$ and therefore $$\dim_F B(2^n)\leq l^2 2^{2n+3} f(2^{n+1}).$$

Since we assume $f(Cn)\geq nf(n)$ we have that $f(2^n)\sim 2^n f(2^n)$, and by Lemma 5.5 from \cite{SmoktunowiczBartholdi}, it follows that $\dim_F B(n)\preceq f(n)$. On the other hand, there is a graded surjective homomorphism $B\rightarrow B_0$, so finally $\dim_F B(n)\sim f(n)$, as desired.

\subsection{Locally nilpotent ideals}
Note that $\tilde{C}(2^n)\subseteq \left<x_1,\dots,x_d\right>$ unless $n=\mu(i)=2^{\nu(i)}$. For a monomial $v$, we denote $w\preceq v$ if $w$ is a subword of $v$. If $T$ is a set of monomials then we write $w\preceq T$ if there exists some $v\in T$ with $w\preceq v$.

Recall that $M=\bigcup_{n\in \mathbb{N}}M(n)$ is the set of all monomials in the free algebra $A$. Let $\lambda:M\rightarrow \mathbb{N}\cup\{0\}$ be the degree with respect to $x_{d+1}$. For a subset $T\subseteq \tilde{W}$ and $m\in \mathbb{N}$ define:
\begin{itemize}
\item $\lambda(m,T)=\max_{w\preceq T, |w|\leq m} \{\lambda(w)\}$;
\item $\lambda(T)=\sup_{m\in \mathbb{N}} \{\lambda(m,T)\}$;
\item $\lambda(m)=\sup_{w\preceq \tilde{W},|w|\leq m} \{\lambda(w)\}$.
\end{itemize}

\begin{lem} \label{prebound}
For every $m\gg 1$ we have: $$\lambda(\tilde{W}(2^m))\leq 2m.$$
\end{lem}
\begin{proof}
Fix $m$. If $m\neq 2^{\nu(i)}$ for all $i$ then $\tilde{W}(2^{m+1})=C(2^m)\tilde{W}(2^m)$ with $\lambda (C(2^m))=0$ (namely, $C(2^m)\subseteq \left<x_1,\dots,x_d\right>$). Therefore, $\lambda(\tilde{W}(2^{m+1}))=\lambda(\tilde{W}(2^m))$.

If $m=2^{\nu(i)}$ then $\lambda(\tilde{W}(2^{m+1}))\leq 2\lambda(\tilde{W}(2^m))$. Recall that $\lambda(\tilde{W}(1))=1$.

Hence for all $m\gg 1$ we have $\lambda(\tilde{W}(2^m))\leq 2^{\lceil \log_2 m\rceil}\leq 2m$.
\end{proof}

\begin{lem} \label{boundL}
For $n\gg 1$, we have $\lambda(n)\leq 2m$ where $m=\lceil \log_2 n\rceil+2$.
\end{lem}
\begin{proof}
Let $w\preceq \tilde{W}$ with $|w|<2^k$, for some $k\neq 2^{\nu(i)}$ (for all $i$). We now show that $\lambda(w)\leq \lambda(\tilde{W}(2^k))$.

Assume $w\preceq \tilde{W}(2^{k+t})$ for some $t>0$. We prove our claim by induction on $t$. For $t=1$, $\tilde{W}(2^{k+1})=\tilde{C}(2^k)\tilde{W}(2^k)$ and $\lambda(\tilde{C}(2^k))=0$ so the claim holds. For $t\geq 2$, if $k+t-1\neq 2^{\nu(i)}$ then $\tilde{W}(2^{k+t})=\tilde{C}(2^{k+t-1})\tilde{W}(2^{k+t-1})$ with $\lambda(\tilde{C}(2^{k+t-1}))=0$, so the claim holds by induction; otherwise, $k+t-1=2^{\nu(i)}$ for some $i$, so $\tilde{W}(2^{k+t})=\tilde{C}(2^{k+t-1})\tilde{C}(2^{k+t-2})\tilde{W}(2^{k+t-2})$ with $\lambda(\tilde{C}(2^{k+t-2}))=0$ (since $k+t-2$ is not a power of $2$), and therefore we may assume either $w\preceq \tilde{C}(2^{k+t-1})\subseteq \tilde{W}(2^{k+t-1})$ or $w\preceq \tilde{W}(2^{k+t-1})$ and in both cases we are done by induction.

Take $k=\lceil \log_2 n\rceil+i$ for $i\in \{1,2\}$ such that $k$ is not of the form $2^{\nu(j)}$. It follows by Lemma \ref{prebound} that $$\lambda(n)\leq \lambda(n,\tilde{W}(2^k))\leq \lambda(n,\tilde{W}(2^m))\leq \lambda(\tilde{W}(2^m))\leq 2m.$$
\end{proof}

\begin{prop}
The ideal $Bx_{d+1}B$ is locally nilpotent.
\end{prop}

\begin{proof}
Assume otherwise. It follows that there exists a finite set of elements $S=\{u_1x_{d+1}v_1,\dots,u_mx_{d+1}v_m\}$ such that for every $n$ there exists $r$ such that $S^n$ contains a subword of some monomial from $\tilde{W}(2^r)$.

Let $M=\max_i \{ |u_i|+|v_i|+1\}$. We have that for every $n$ there exists a non-zero monomial $w_n\in S^n$ with $|w_n|\leq M\cdot n$ and $\lambda(w_n)\geq n$. Fix some $n\gg 1$.

Then, by Lemma \ref{boundL}: $n\leq \lambda(w_n)\leq \lambda(M\cdot n)\leq 2(\lceil \log_2(M\cdot n)\rceil+2)$, a contradiction.
\end{proof}

\section{Concluding Remarks}

We finish with the following remarks and questions which are related to the construction in Section \ref{section2}.

\begin{ques}
What are the possible growth types of finitely generated simple algebras?
\end{ques}

We also note that in \cite{Groupoid}, finitely generated algebras were constructed out of infinite words; these algebras are simple provided that the word, say, $w$ is aperiodic and uniformly recurrent (namely, for every finite subword $u$ of $w$ there exists some $r>0$ for which every subword of $w$ with length at least $r$ contains a copy of $u$). 

Moreover, it was shown that the growth of the resulting simple algebra is $\sim np_w(n)$ where $p_w(n)$ is the complexity function of $w$, counting subwords of $w$ of length $n$.

In \cite{Julien}, infinite (aperiodic) uniformly recurrent words were introduced, having prescribed complexity function which can be taken to be from a relatively wide variety of monotone submultiplicative functions (see Theorem 3 in \cite{Julien}).

However, note that the growth functions constructed there are still not general enough to include all possible growth functions (e.g. $p(n)=2^{n/\log n}$ does not satisfy condition (iii) in Theorem 3 in \cite{Julien}, which requires that the derivative of the exponent, namely $\frac{1}{\log n}-\frac{1}{\log^2 n}$ tends to zero at least as quick as $n^{-\beta}$ for some $\beta>0$).

It is possible to construct finitely generated domains of intermediate growth. Indeed, if $\mathfrak{A}$ is a Lie algebra of polynomial growth then its universal enveloping algebra $U(\mathfrak{A})$ has subexponential, but not polynomially bounded growth. For details, see \cite{Lie}.

Here is another, group-theoretic based, example of a domain with intermediate growth. In \cite{domain_int}, Grigorchuk constructed a torsion-free group $G$ with intermediate growth, which turns out to be right-ordered. By a result of Passman (see Lemma 13.1.9 in \cite{Passman}), this implies that the group algebra $F[G]$ is a domain, and therefore a domain of intermediate growth.

We ask the following:

\begin{ques}
What are the possible growth types of finitely generated domains?
\end{ques}

In particular, it would be interesting to know if there exists a domain (or: graded domain) with super-polynomial growth function $f$, such that $f(n)\prec \exp(n^{\alpha})$, for every $\alpha>0$.


\end{document}